\documentclass{amsart}

\usepackage{hyperref}
\usepackage{color}
\usepackage{ulem}

\newtheorem{theorem}{Theorem}[section]

\newtheorem{remark}[theorem]{Remark}

\begin{document}

\title[Miscoscopic expression of anomalous dissipation
]
{Microscopic expression of anomalous dissipation in Passive scalar transport
} 
\author{Tomonori Tsuruhashi} 
\address{Graduate School of Mathematical Sciences, University of Tokyo, Komaba 3-8-1 Meguro, 
Tokyo 153-8914, Japan,
} 
\email{
tomonori@ms.u-tokyo.ac.jp} 

\author{Tsuyoshi Yoneda} 
\address{Graduate School of Economics, Hitotsubashi University, 2-1 Naka, Kunitachi, Tokyo 186-8601, Japan} 
\email{t.yoneda@r.hit-u.ac.jp}





\begin{abstract}
We study anomalous dissipation from a microscopic viewpoint. 
In the work by Drivas, Elgindi, Iyer and Jeong (2022), the property of anomalous dissipation provides the existence of non-unique weak solutions for a transport equation with a singular velocity field. 
In this paper, we reconsider this solution in terms of kinetic theory and clarify
its microscopic property. 
Consequently, energy loss can be expressed by non-vanishing microscopic advection. 

\end{abstract} 

\maketitle

\section{Introduction}

In this paper, we consider the transport equation
\begin{equation}\label{tp}
\partial_t \theta + u \cdot \nabla \theta = 0, \quad  \theta|_{t=0}=\theta_0 
\qquad \mathrm{in}  \quad  [0, T] \times\mathbb{T}^d,   
\end{equation}
where $\theta$ is a real-valued function, $u$ is a divergence-free vector-valued function, $T > 0$ and $\mathbb{T}^d$ is defined by $\mathbb{R}^d / \mathbb{Z}^d$.
It is known that a weak solution of $(\ref{tp})$ can be constructed as a vanishing viscosity limit via the advection-diffusion equation:
\begin{equation}\label{ad}
\partial_t \theta^\kappa + u \cdot \nabla \theta^\kappa = \kappa \Delta \theta^\kappa, 
\quad \theta^\kappa |_{t=0}=\theta_0  
\qquad \mathrm{in}  \quad  [0, T] \times \mathbb{T}^d, 
\end{equation}
where $\kappa > 0$ and $\theta^\kappa$ is a real-valued function. 
Recently,  Drivas, Elgindi, Iyer and Jeong \cite{MR4381138} have shown the non-uniqueness results on $(\ref{tp})$ provided by
the anomalous dissipation:
\begin{align} \label{ano-dis}
\kappa \int_0 ^t \| \nabla \theta^\kappa(s) \|_{L^2(\mathbb{T}^d)}^2 ds \ge c > 0, 
\end{align}
where $c > 0$ is a constant independent of $\kappa$.
Concerning uniqueness, the regularity of $\theta$ and $u$ is crucial. 
Since the pioneering work of the renormalized solution by DiPerna and Lions \cite{MR1022305}, it has been extensively studied the relation between the regularity and the uniqueness (see for instance \cite{MR2096794,MR4071413}). 
Note that, considering the results in \cite{MR1022305}, we can see that the non-unique weak solutions in \cite{MR4381138} are not renormalized. 
According to the kinetic theory,  
construction of a global solution of first order quasi-linear equations was initially done by \cite{MR765504,MR705037,MR768737}. 
With the aid of
this theory,
 the first author \cite{MR4490734} characterized renormalization from a microscopic viewpoint and showed the violation of the chain rule, which causes the failure of renormalization.

Now, let us reformulate $(\ref{ad})$ from a microscopic viewpoint.
Based on the idea of the previous result \cite{MR4490734}, we expect that a renormalized solution of $(\ref{ad})$ satisfies the following equation formally: 
\begin{equation}\label{re-ad}
\partial_t \Phi( \theta^\kappa) + u \cdot \nabla \Phi( \theta^\kappa) = \kappa \Delta \Phi(\theta^\kappa) - \kappa \Phi''(\theta^\kappa) \sum_{i = 1}^d (\partial_{x_i} {\theta^\kappa})^ 2  \qquad  \mathrm{in}  \quad  [0, T] \times \mathbb{T}^d .
\end{equation}
Here, $\Phi: \mathbb{R} \to \mathbb{R}$ is a smooth function.
Indeed, if $\theta^\kappa$ is smooth and satisfies $(\ref{ad})$, it holds that
\begin{align*}
&\partial_t \Phi( \theta^\kappa) + u \cdot \nabla \Phi( \theta^\kappa) - \kappa \Delta \Phi(\theta^\kappa) + \kappa \Phi''(\theta^\kappa) \sum_{i = 1}^d (\partial_{x_i} {\theta^\kappa})^ 2   \\
=& \Phi'( \theta^\kappa) \partial_t  \theta^\kappa + u \cdot \Phi'( \theta^\kappa) \nabla \theta^\kappa\\
& \qquad  - \kappa \{ \Phi''(\theta^\kappa) \sum_{i = 1}^d (\partial_{x_i} {\theta^\kappa})^ 2 + \Phi'(\theta^\kappa) \Delta \theta^\kappa \}  + \kappa \Phi''(\theta^\kappa) \sum_{i = 1}^d (\partial_{x_i} {\theta^\kappa})^ 2   \\
=& \Phi'( \theta^\kappa) \{ \partial_t  \theta^\kappa + u \cdot \nabla \theta^\kappa - \kappa \Delta \theta^\kappa \} \\
=& 0.
\end{align*}
We call a solution of $(\ref{ad})$ is renormalized when it satisfies $(\ref{re-ad})$ for all $\Phi \in X$, where $X = \{ \int_{- \infty}^\tau \varphi(\xi) d\xi | \varphi \in C_0^\infty(\mathbb{R}) \}$ and $C_0^\infty(\mathbb{R})$ consists of all smooth functions which have a compact support in $\mathbb{R}$.
We regard the equation $(\ref{ad})$ as a macroscopic expression via the following form:
$$\chi_{\theta^\kappa} (t, x, \xi) = \left\{
\begin{split}
&1 &\quad &\left( \, \xi \le \theta^\kappa(t, x) \, \right), \\
&0 &\quad &\left( \, \xi > \theta^\kappa(t, x) \, \right).
\end{split} \right. $$
Then, the microscopic equation of $(\ref{ad})$ can be expressed as follows:
\begin{equation}\label{k-ad}
\partial_t \chi_{\theta^\kappa} + u \cdot \nabla \chi_{\theta^\kappa}
 = - \kappa \Delta \chi_{\theta^\kappa} + \kappa \sum_{i = 1}^d (\partial_{x_i} {\theta^\kappa})^ 2  \partial_\xi^2 \chi_{\theta^\kappa} \qquad  \mathrm{in}  \quad  [0, T] \times \mathbb{T}^d \times \mathbb{R}.
\end{equation}
Note that a weak solution to $(\ref{k-ad})$ is a renormalized solution to $(\ref{ad})$. 
In the same way, the microscopic equation of  $(\ref{tp})$ can be expressed as the following:
\begin{equation}\label{k-tp}
\partial_t \chi_\theta + u \cdot \nabla \chi_\theta
 = 0 \qquad  \mathrm{in}  \quad  [0, T] \times \mathbb{T}^d \times \mathbb{R}.
\end{equation}
By describing the property of anomalous dissipation and the non-uniqueness result in \cite{MR4381138} from the microscopic viewpoint, we will show that anomalous dissipation yields a microscopic external force in terms of the kinetic theory.
To be precise, in the microscopic viewpoint, the corresponding weak solution does not satisfy the desired microscopic equation $(\ref{k-tp})$, which is caused by the remaining  of  the term 
$\kappa \sum_{i = 1}^d (\partial_{x_i} {\theta^\kappa})^ 2  \partial_\xi^2 \chi_{\theta^\kappa}$. 

To begin with, we take a closer look of the quantity on $(\ref{ano-dis})$. 
Anomalous dissipation leads to the existence of a non-zero limit for the microscopic advection term. 
\begin{theorem} \label{non-zero-distr}
Let $0 < \tau < T$ and $0 \leq \alpha < 1$. 
Suppose that $\theta_0$ be smooth and that 
$$u \in C^\infty(([0, \tau) \cap (\tau, T)) \times \mathbb{T}^d)\cap L^1(0,T:C^\alpha(\mathbb{T}^d))\cap L^\infty([0,T)\times\mathbb{T}^d).$$ 
Let $\{ \theta^\kappa \}_{\kappa > 0}$ be a sequence of weak solutions to $(\ref{ad})$. 
If there exists a constant $c$ independent of $\kappa$ such that 
\begin{align} \label{anomalous-ineq}
\kappa \int_0 ^\tau \| \nabla \theta^\kappa(s) \|_{L^2(\mathbb{T}^d)}^2 ds \ge c > 0
\end{align}
for all $\kappa > 0$, then the term $\kappa \sum_{i = 1}^d (\partial_{x_i} {\theta^\kappa})^ 2  \partial_\xi^2 \chi_{\theta^\kappa}$ 
in $(\ref{k-ad})$ does not go to $0$ as $\kappa$ goes to $0$ in the sense of distributions.
\end{theorem}
This theorem says that the existence of energy loss in some time interval yields a microscopic external force in the equation $(\ref{k-tp})$. 
More precisely, an inviscid limit of solutions to the advection equation gives a weak solution to the transport equation from a macroscopic viewpoint,  however, it is not adequate from a microscopic viewpoint.

Of course, the next question is naturally raised; whether there really exists such a sequence which gives a non-zero limit on a  microscopic advection term in Theorem $\ref{non-zero-distr}$. 
To answer this question, we construct such a sequence in almost the same way as in \cite{MR4381138}
(we just technically need to apply a mollifier to the sequence of solutions to be smooth, see also Remark 3.1 in \cite{MR4381138}). 
By adopting the velocity so that the assumptions of Theorem 3 in \cite{MR4381138} are fulfilled, we have a result on the existence of non-zero limit for the microscopic advection term. 

\begin{theorem} \label{result-like-DEIJ}
Let $d \ge 2$, $0 \leq \alpha < 1$ and $\theta_0 \in H^2(\mathbb{T}^d)$ be a mean-zero initial data for the equation $(\ref{tp})$. 
Then, there exist a divergence-free vector field 
$$u_{*} \in C^\infty([0,T/2)\times\mathbb{T}^d)\cap L^1([0,T/2]:C^\alpha(\mathbb{T}^d))\cap L^\infty([0,T/2]\times\mathbb{T}^d)$$ 
such that the following holds: 
Set 
$$u(t) = \left\{
\begin{split}
&u_{*}(t) &\quad &t \in [0, T/2), \\
&-u_{*}(T - t) &\quad &t \in [T/2, T]. 
\end{split} \right. $$
Then, there exists an inviscid solution to $(\ref{tp})$ which does not satisfy the equation $(\ref{k-tp})$. 
\end{theorem}

We emphasize that the microscopic advection term does not go to zero and that it remains as a microscopic external force in the right-hand side of the equation $(\ref{k-tp})$. 
On the other hand, such an external force must vanish in a smooth setting.

\begin{remark}
Let $\theta_0$ and $u$ be smooth. 
Suppose that $\theta$ be a smooth solution to $(\ref{tp})$, which is given by an inviscid limit of smooth solutions $\{ \theta^\kappa \}_{\kappa > 0}$ to $(\ref{ad})$.
Then, $\theta$ satisfies the microscopic equation $(\ref{k-tp})$. 
\end{remark}

In particular, it is worth pointing out that the microscopic advection term 
$$\kappa \sum_{i = 1}^d (\partial_{x_i} {\theta^\kappa})^ 2  \partial_\xi^2 \chi_{\theta^\kappa}$$ 
goes to zero as $\kappa$ goes to zero in the smooth setting. 
Although the proof of this remark is straightforward by our microscopic definition, 
the comparison between Theorem \ref{result-like-DEIJ} and this remark clarifies the influence of anomalous dissipation from a microscopic viewpoint. 
Consequently, energy loss can be expressed by non-vanishing microscopic advection.   


\section{Proof of theorems}

\begin{proof}[Proof of Theorem $\ref{non-zero-distr}$]

It is sufficient to show that there exist a positive constant $C>0$ and an appropriate test function $\phi \in C^\infty([0, T] \times \mathbb{T}^d \times \mathbb{R}^d)$ with $\mathrm{supp} \, \phi \subset [0, T) \times \mathbb{T}^d \times \mathbb{R}^d$ to be compact such that
$$\int_0^T\int_{\mathbb{T}^d} \int_{\mathbb{R}} \kappa \sum_{i = 1}^d (\partial_{x_i} {\theta^{\kappa}})^ 2 \chi_{\theta^{\kappa}} \partial_\xi^2 \phi  d\xi dx dt \ge C$$ 
for all $\kappa > 0$. 

Recall that the definition of $\chi_{\theta^\kappa}$ be 
$$\chi_{\theta^\kappa} (t, x, \xi) = \left\{
\begin{split}
&1 &\quad &\left( \, \xi \le \theta^\kappa(t, x) \, \right), \\
&0 &\quad &\left( \, \xi > \theta^\kappa(t, x) \, \right).
\end{split} \right. $$

Since $\{\theta^\kappa \}_\kappa$ is a sequence of solutions to $(\ref{ad})$, we can assume without the loss of generality that $-1\leq \theta^\kappa\leq 1$ on $[0, \tau) \times \mathbb{T}^d$ by maximal principal. 
Let $\tilde{\phi}: [0, T) \times \mathbb{T}^d \times \mathbb{R}^d \to \mathbb{R}$ be taken so that  
\begin{equation*}
\partial_\xi^2\tilde{\phi}(t, x, \xi)=
\begin{cases}
-\epsilon,\quad \xi\in[-3-\frac{1}{\epsilon},-3),\\
1,\quad \xi\in[-3,-1),\\
-\epsilon,\quad \xi\in [-1,-1+\frac{1}{\epsilon}),\\
0,\quad \text{otherwise}
\end{cases}
\end{equation*}
where $\epsilon > 0$ is a constant independent of $\kappa$ with $-1+1/\epsilon > 1$. 
%
Then, it holds that 
\begin{equation*}
\int_{\mathbb{R}} \chi_{\theta^{\kappa}} \partial_\xi^2 \tilde{\phi}(t, x, \xi)  d\xi
\ge 1 - 2 \epsilon > 0  \quad \mathrm{in} \, \, [0, \tau) \times \mathbb{T}^d
\end{equation*}
for all $\kappa$.
We take $\phi$ as a test function by modifying $\tilde{\phi}$ to be smooth and to satisfy its support condition in the direction $t$. 
In what follows, $C$ denotes a constant independent of $\kappa > 0$, whose exact value may change. 
By taking $\epsilon$ smaller if necessary, it holds that 
\begin{equation*}
\int_0^T\int_{\mathbb{T}^d} \int_{\mathbb{R}} \kappa \sum_{i = 1}^d (\partial_{x_i} {\theta^{\kappa}})^ 2 \chi_{\theta^{\kappa}} \partial_\xi^2 \phi  d\xi dx dt
\geq  C \int_0^\tau\int_{\mathbb{T}^d} \kappa \sum_{i = 1}^d (\partial_{x_i} {\theta^{\kappa}})^ 2 dx dt.
\end{equation*}
%
By the inequality $(\ref{anomalous-ineq})$, we have 
\begin{equation*}
\int_0^T\int_{\mathbb{T}^d} \int_{\mathbb{R}} \kappa \sum_{i = 1}^d (\partial_{x_i} {\theta^{\kappa}})^ 2 \chi_{\theta^{\kappa}} \partial_\xi^2 \phi  d\xi dx dt
\geq C > 0. 
\end{equation*}
This completes the proof.

\end{proof}

\begin{proof}[Proof of Theorem $\ref{result-like-DEIJ}$]

Let $\{\theta^\kappa \}$ be a solution to $(\ref{ad})$, which is the same as the example in \cite{MR4381138}. 
We modify this sequence so that it is smooth in the interval $(0, T/2)$ and $(T/2, T)$. 
To simplify notation, we use the same letter $\{\theta^\kappa \}$ for this modified sequence. 

Note that there exists an inviscid solution to $(\ref{tp})$ by the construction of $\{\theta^\kappa \}$. 
We denote by $\theta$ this inviscid solution. 
Since $\{\theta^\kappa \}$ be a solution to $(\ref{ad})$, we have
\begin{align} \label{weak-ad}
0 = \int_0^T\int_{\mathbb{T}^d} \int_{\mathbb{R}} & \{  \chi_{\theta^\kappa} \partial_t \phi + u \cdot \chi_{\theta^\kappa} \nabla_x \phi  - \kappa \chi_{\theta^\kappa} \Delta \phi + \kappa \sum_{i = 1}^d (\partial_{x_i} {\theta^\kappa})^ 2 \chi_{\theta^\kappa} \partial_\xi^2 \phi \} d\xi dx dt \\
& + \int_{\mathbb{T}^d} \int_{\mathbb{R}}  \chi_{\theta_0} \, \phi|_{t=0} d\xi dx
\end{align}
for all $\phi \in C^\infty([0, T] \times \mathbb{T}^d \times \mathbb{R}^d)$ with $\mathrm{supp} \, \phi \subset [0, T) \times \mathbb{T}^d \times \mathbb{R}^d$ to be compact.

By the definition of $\chi_{\theta^\kappa}$, it holds that
\begin{equation*}
\kappa\int_0^T\int_{\mathbb{T}^d}\int_{\mathbb{R}}\chi_{\theta^\kappa} \Delta\phi d\xi dxdt\to 0 
\quad \mathrm{as} \quad \kappa \to 0.
\end{equation*}
On the other hand, $\theta^\kappa$ satisfies the inequality $(\ref{anomalous-ineq})$ with $\tau$ to be $T/2$ by the construction in \cite{MR4381138}. 
In a similar way as the proof of Theorem $\ref{non-zero-distr}$, we see that 
the term $\kappa \sum_{i = 1}^d (\partial_{x_i} {\theta^\kappa})^ 2  \partial_\xi^2 \chi_{\theta^\kappa}$ does not go to $0$ as $\kappa$ goes to $0$ in the sense of distributions. 

These calculations show that the inviscid solution $\theta$ cannot give a solution to $(\ref{k-tp})$. 
This completes the proof.

\end{proof}

\vspace{0.5cm}
\noindent
{\bf Acknowledgments.}\ 
We would like to thank Professor Yoshikazu Giga for valuable comments on this work. 
Research of TT was partly supported by the JSPS Grant-in-Aid for JSPS Research Fellows 22J10745.
Research of TY was partly supported by the JSPS Grants-in-Aid for Scientific
Research  20H01819. 


\end{document}